\documentclass[10pt; a4paper]{amsart} 
\usepackage{amscd, amsfonts, amsmath, amssymb, amsthm, fixmath, mathrsfs}
\usepackage[all]{xy} 
\makeindex

\theoremstyle{definition} 
\newtheorem{Unity}{Unity}[section] 
\newtheorem*{Definition*}{Definition} 
\newtheorem{Definition}[Unity]{Definition} 

\theoremstyle{plain} 
\newtheorem*{Theorem*}{Theorem}
\newtheorem{Theorem}[Unity]{Theorem}
\newtheorem{Proposition}[Unity]{Proposition}
\newtheorem{Corollary}[Unity]{Corollary}
\newtheorem{Lemma}[Unity]{Lemma}

\theoremstyle{remark} 
\newtheorem*{Remark*}{Remark}
\newtheorem{Remark}[Unity]{Remark}

\numberwithin{Unity}{section}

\newcommand{\Z}{\mathbb{Z}}
\newcommand{\R}{\mathbb{R}}
\newcommand{\E}{{\mathscr E}}
\newcommand{\F}{{\mathscr F}}
\newcommand{\Ox}{{\mathscr O}}
\newcommand{\Pg}{{\mathscr P}}
\newcommand{\HNP}{\mathrm{HNP}}

\newcommand{\Spec}{\mathrm{Spec}}
\newcommand{\Jac}{\mathrm{Jac}}

\newcommand{\rank}{\mathrm{rk}}
\newcommand{\degree}{\mathrm{deg}}
\newcommand{\Omg}{\mathrm{\Omega}}
\newcommand{\ConPg}{\frak{ConPg}} 

\begin{document}

\title{The Morphism Induced by Frobenius Push-Forwards}
\author{Lingguang Li}
\address{School of Mathematical Sciences, Fudan University, Shanghai, P. R. China}
\email{LG.Lee@amss.ac.cn}
\maketitle

\begin{abstract} Let $X$ be a smooth projective curve of genus $g(X)\geq 1$ over an algebraically
closed field $k$ of
characteristic $p>0$ and $F_{X/k}:X\rightarrow X^{(1)}$ be the relative Frobenius morphism.
Let $\mathfrak{M}^{s(ss)}_X(r,d)$ (resp. $\mathfrak{M}^{s(ss)}_{X^{(1)}}(r\cdot p,d+r(p-1)(g-1))$)
be the moduli space of (semi)-stable vector bundles of rank $r$ (resp. $r\cdot p$) and degree $d$
(resp. $d+r(p-1)(g-1)$) on $X$ (resp. $X^{(1)}$). We show that the set-theoretic map
$S^{ss}_{\mathrm{Frob}}:\mathfrak{M}^{ss}_X(r,d)\rightarrow\mathfrak{M}^{ss}_{X^{(1)}}(r\cdot p,d+r(p-1)(g-1))$
induced by $[\E]\mapsto[{F_{X/k}}_*(\E)]$ is a proper morphism. Moreover, if $g(X)\geq 2$,
the induced morphism
$S^s_{\mathrm{Frob}}:\mathfrak{M}^s_X(r,d)\rightarrow\mathfrak{M}^s_{X^{(1)}}(r\cdot p,d+r(p-1)(g-1))$
is a closed immersion. As an application, we obtain that the locus of moduli space $\mathfrak{M}^{s}_{X^{(1)}}(p,d)$ consists of stable vector bundles whose Frobenius pull back have maximal Harder-Narasimhan Polygon is isomorphic to Jacobian variety $\Jac_X$ of $X$.
\end{abstract}

\section{Introduction}

Let $k$ be an algebraically closed field of characteristic $p>0$, $X$ a smooth projective curve of genus $g$ over $k$.
The absolute Frobenius morphism $F_X:X\rightarrow X$ is induced by
$\Ox_X\rightarrow \Ox_X$, $f\mapsto f^p$. Let $F_{X/k}:X\rightarrow X^{(1)}:=X\times_kk$
denote the relative Frobenius morphism of $X$ over $k$. Let $\mathfrak{M}^{s(ss)}_X(r,d)$ be the moduli space of
(semi)-stable vector bundles of rank $r$ and degree $d$ on $X$.

It is well known that the (semi)-stability of vector bundles is not preserved by Frobenius pull back $F^*_X$. Therefore,
the set-theoretic map
$$V_{r,d}:\mathfrak{M}^{ss}_X(r,d)\dashrightarrow \mathfrak{M}^{ss}_X(r,pd)$$
$$[\E]\mapsto[F^*_X(\E)]$$
is not well-defined on the whole moduli space
$\mathfrak{M}^{ss}_X(r,d)$. Denote
$$U^{ss}_X(r,d):=\{[\E]\in\mathfrak{M}^{ss}_X(r,d)~|~F^*_X(\E)~\text{is
a semi-stable vector bundle}\}.$$ Then
$U^{ss}_X(r,d)\subseteq\mathfrak{M}^{ss}_X(r,d)$ is an open
sub-variety, and $V_{r,d}|_{U^{ss}_X(r,d)}:U^{ss}_X(r,d)\rightarrow
\mathfrak{M}^{ss}_X(r,pd)$ is a well defined morphism (See the
proof of \cite[Proposition 9]{Osserman06}). The rational
map $V_{r,d}$ is called \emph{generalized Verschiebung rational
map}.

On the other hand, the (semi)-stability of Frobenius direct image has been study by many mathematicians.
\begin{itemize}
\item[(i)] H. Lange and C. Pauly \cite[Proposition 1.2]{LangePauly08} showed that if $g\geq 2$, ${F_{X/k}}_*(\mathscr{L})$ is stable on $X^{(1)}$ for any line bundle $\mathscr{L}$ on $X$;
\item[(ii)] V. Mehta and C. Pauly \cite[Theorem 1.1]{MehtaPauly07} proved that if $g\geq 2$, then for any semi-stable bundle $\E$ on $X$, ${F_{X/k}}_*(\E)$ is also semistable.
\item[(iii)] X. Sun \cite[Theorem 2.2]{Sun08} showed that if $g\geq 1$, than ${F_{X/k}}_*(\E)$ is semi-stable whenever $\E$ is semi-stable on $X$. Moreover, if $g\geq 2$, then ${F_{X/k}}_*(\E)$ is stable whenever $\E$ is stable.
\end{itemize}

In section 3, we show that the set-theoretic map
\begin{eqnarray*}
S^{ss}_{\mathrm{Frob}}:\mathfrak{M}^{ss}_X(r,d)&\rightarrow&\mathfrak{M}^{ss}_{X^{(1)}}(r\cdot p,d+r(p-1)(g-1))\\
\/[\E\/]&\mapsto&\/[{F_{X/k}}_*(\E)\/]
\end{eqnarray*}
is a  proper morphism. Moreover, if $g(X)\geq 2$, the induced morphism $$S^s_{\mathrm{Frob}}:\mathfrak{M}^s_X(r,d)\rightarrow\mathfrak{M}^s_{X^{(1)}}(r\cdot p,d+r(p-1)(g-1))$$ is a closed immersion (Theorem \ref{Lee}).

Let $\E$ a vector bundle on $X$, we consider the Harder-Narasimhan filtration of $\E$
$$\E^{\mathrm{HN}}_{\bullet}:0=\E_m\subset\E_{m-1}\subset\cdots\subset\E_1\subset\E_0=\E.$$
For any subbundle $\E_i$, we may associate to it the point $(\rank(\E_i),\degree(\E_i))$ in the plane $\R^2$, and we connect point $(\rank(\E_i),\degree(\E_i))$ to
point $(\rank(\E_{i-1}),\degree(\E_{i-1}))$ successively by line segments for $0\leq i\leq m$. Then we get a convex polygon in the plane $\R^2$ which we call the \emph{Harder-Narasimhan Polygon} of $\E$, denote by $\HNP(\E)$.

Let $\ConPg(r,d)$ be the category of convex polygons with starting point at $(0,0)$ and terminal point $(r,d)$, there is a natural partial order structure, denote by "$\succcurlyeq$", on $\ConPg(r,d)$. Let $\Pg_1,\Pg_2\in\ConPg(r,d)$, we say $\Pg_1\succcurlyeq\Pg_2$ if and only if $\Pg_1$ lies on or above $\Pg_2$.

Consider the natural map
\begin{eqnarray*}
\mathfrak{M}^{s}_{X^{(1)}}(r,d)&\rightarrow&\ConPg(r,pd)\\
\E&\mapsto&\HNP(F^*_{X/k}(\E))
\end{eqnarray*}
There is a canonical stratification (\emph{Frobenius stratification}) on $\mathfrak{M}^{s}_X(r,d)$ by Harder-Narasimhan polygons \cite{JRXY06}. By a theorem of S. S. Shatz \cite[Theorem 3]{Shatz77}, the subset $$S_{\Pg}:=\{\E\in\mathfrak{M}^{s}_{X^{(1)}}(r,d)|\HNP(F^*_{X/k}(\E))\succcurlyeq\Pg\}$$
is a closed subvariety of $\mathfrak{M}^{s}_X(r,d)$ for any $\Pg\in\ConPg(r,pd)$.

The fundamental question of the Frobenius stratification of muduli space of (semi)-stable vector bundles in positive characteristic is: what is the geometric properties, such as non-emptiness, irreducibility, smoothness, dimension and so on, of each stratum of Frobenius stratification. However, very little is known about the strata of Frobenius stratification. Some results are only known in special cases or for small values of $p$, $g$, $r$ and $d$. For example, K. Joshi, S. Ramanan, E. Z. Xia and J.-K. Yu \cite{JRXY06} given a completely understand of moduli space $\mathfrak{M}^{s}_{X^{(1)}}(2,d)$ when $p=2$ and $g\geq 2$. They proved the irreducibility of each non-empty stratum of Frobenius stratification and obtained their respective dimensions. This is the only case understood completely.

In section 4, we show that a stable bundles $\E$ of rank $p$ and degree $d$ on $X^{(1)}$ whose Frobenius pull back $F^*_{X/k}(\E)$ have maximal Harder-Narasimhan Polygon if and only if $\E\cong{F_{X/k}}_*(\mathscr{L})$ for some line bundle $\mathscr{L}$ on $X$. It follows that the locus
$$W=\{\E\in\mathfrak{M}^{s}_{X^{(1)}}(p,d)~|~\HNP(F^*_{X/k}(\E))\succcurlyeq\HNP(F^*_{X/k}(\F))~\text{for any}~\F\in\mathfrak{M}^{s}_{X^{(1)}}(p,d)\}$$
of moduli space $\mathfrak{M}^{s}_{X^{(1)}}(p,d)$ is precisely the image of the morphism $$S^s_{\mathrm{Frob}}:\mathfrak{M}^{s}_X(1,d-(p-1)(g-1))\rightarrow\mathfrak{M}^{s}_{X^{(1)}}(p,d).$$
Hence, $W$ is a closed sub-variety of $\mathfrak{M}^{s}_{X^{(1)}}(p,d)$ which is isomorphic to Jacobian variety $\Jac_X$ of $X$ (Corollary \ref{MaxFrobStr}). This result generalize the partial result of \cite[Theorem 4.6.4]{JRXY06} via different method.

\section{Canonical Connection and Canonical Filtration}

Let $k$ be an algebraically closed field of characteristic $p>0$,
and $X$ a smooth projective variety over $k$.
Consider the commutative diagram
$$\xymatrix{
  X \ar@/_/[ddr]_{\pi} \ar@/^/[drr]^{F_X} \ar@{.>}[dr]|-{F_{X/k}}\\
   & X^{(1)} \ar[d]^{\pi^{(1)}} \ar[r] & X \ar[d]^{\pi}\\
   & \Spec(k) \ar[r]^{F_k}             & \Spec(k).}$$
For any coherent sheaf $\F\in\mathfrak{Qco}(X^{(1)})$, there exists a \emph{canonical connection} on the coherent sheaf
$F^*_{X/k}(\F)$, denote by $(F^*_{X/k}(\F),\nabla_{\mathrm{can}})$:
$$\nabla_{\mathrm{can}}:F^*_{X/k}(\F)\rightarrow F^*_{X/k}(\F)\otimes_{\Ox_X}\Omg^1_{X/k}$$ locally defined by $f\otimes m\mapsto m\otimes d(f)$, where $m\in\F$, $f\in\Ox_X$,
$d:\Ox_X\rightarrow\Omg^1_{X/k}$ is the canonical exterior differentiation.

\begin{Definition}\label{CanFil}
Let $k$ be an algebraically closed field of characteristic $p>0$, and $X$ a smooth projective variety over $k$. For
any coherent sheaf $\E$ on $X$, let
$$\nabla_{\mathrm{can}}:F^*_{X/k}{F_{X/k}}_*(\E)\rightarrow
F^*_{X/k}{F_{X/k}}_*(\E)\otimes_{\Ox_X}\Omg^1_{X/k}$$
be the canonical connection on $F^*_{X/k}{F_{X/k}}_*(\E)$.
Set
\begin{eqnarray*}
V_0&:=&F^*_{X/k}{F_{X/k}}_*(\E),\\
V_1&:=&\mathrm{ker}(F^*_{X/k}{F_{X/k}}_*(\E)\twoheadrightarrow\E)),\\
V_{l+1}&:=&\mathrm{ker}\{V_l\stackrel{\nabla}{\rightarrow}F^*_{X/k}{F_{X/k}}_*(\E)
\otimes_{\Ox_X}\Omg^1_{X/k}\rightarrow (F^*_{X/k}{F_{X/k}}_*(\E)/V_l)
\otimes_{\Ox_X}\Omg^1_{X/k}\}
\end{eqnarray*}
The filtration
$${\mathbb{F}^{\mathrm{can}}_\E}_\bullet:F^*_{X/k}{F_{X/k}}_*(\E)=V_0\supset V_1\supset V_2\supset\cdots$$
is called the \emph{canonical filtration} of
$F^*_{X/k}{F_{X/k}}_*(\E)$.
\end{Definition}

X. Sun proved the following theorem in \cite{Sun08}.
\begin{Theorem}\cite[Theorem 3.7]{Sun08} Let $k$ be an algebraically closed field of characteristic $p>0$, and
$X$ a smooth projective variety of dimension $n$ over $k$. Let
$\E$ be a vector bundle. Then the canonical filtration of
$F^*_{X/k}{F_{X/k}}_*(\E)$ is
$$0=V_{n(p-1)+1}\subset V_{n(p-1)}\subset\cdots\subset V_1\subset V_0=
F^*_{X/k}{F_{X/k}}_*(\E)$$ with
$\nabla^l:V_l/V_{l+1}\cong\E\otimes_{\Ox_X}\mathrm{T}^l(\Omg^1_{X/k}),~0\leq
l\leq n(p-1).$
\end{Theorem}

Let $\E$ be a vector bundle of rank $n$ on a variety, then
$\mathrm{T}^l(\E)\subset\E^{\otimes l}$ is defined to be the associated vector bundle
of the frame bundle of $\E$ (principal $\mathrm{GL}_n(k)$-bundle) through the representation $\mathrm{T}^l(V)$ (See \cite[Definition 3.4]{Sun08}).

\section{The Morphism Induced by Frobenius Push-Forwards}

In this section, we will study the natural morphism between moduli spaces of (semi)-stable bundles on curves induced
by Frobenius push-forwards.

\begin{Proposition}\label{Prop:SunMorp}
Let $k$ be an algebraically closed field of characteristic $p>0$, and $X$ a smooth projective curve of genus $g\geq 1$ over
$k$. Then
\begin{itemize}
\item[$(1)$.] If $g\geq 1$, then the set-theoretic map
\begin{eqnarray*}
S^{ss}_{\mathrm{Frob}}:\mathfrak{M}^{ss}_X(r,d)&\rightarrow&\mathfrak{M}^{ss}_{X^{(1)}}(r\cdot p,d+r(p-1)(g-1))\\
\/[\E\/]&\mapsto&\/[{F_{X/k}}_*(\E)\/]
\end{eqnarray*}
is a proper morphism.
\item[$(2)$.] If $g\geq 2$, then the morphism $S^{ss}_{\mathrm{Frob}}$ restrict to sub-variety
$\mathfrak{M}^s_X(r,d)$ induces a proper morphism
$S^s_{\mathrm{Frob}}:\mathfrak{M}^s_X(r,d)\rightarrow\mathfrak{M}^s_{X^{(1)}}(r\cdot p,d+r(p-1)(g-1))$.
\end{itemize}
\end{Proposition}

\begin{proof} Let $T$ be an algebraic variety over $k$, and $\mathcal{E}\in\mathfrak{Coh}(T\times X)$
a flat family of semi-stable bundles on $X$ of rank $r$ and degree
$d$ parameterized by $T$. Consider the morphism $1_T\times
F_{X/k}:T\times X\rightarrow T\times X^{(1)}$. Then $(1_T\times
F_{X/k})_*(\mathcal{E})$ is flat over $T$, and $((1_T\times
F_{X/k})_*(\mathcal{E}))_t\cong {F_{X/k}}_*(\mathcal{E}_t)$ for any
$t\in T$. If $g\geq 1$, then ${F_{X/k}}_*(\mathcal{E}_t)$ are
semi-stable vector bundles of rank $r\cdot p$ and degree
$d+r(p-1)(g-1)$ by \cite[Theorem 2.2, Lemma 4.2]{Sun08}. Thus,
$(1_T\times F_{X/k})_*(\mathcal{E})$ is a flat family of semi-stable
bundles of rank $r\cdot p$ and degree $d+r(p-1)(g-1)$ parameterized
by $T$. Hence, by the universal property of
$\mathfrak{M}^{ss}_{X^{(1)}}(r\cdot p,d+r(p-1)(g-1))$, the set-theoretic map
\begin{eqnarray*}
V_{\mathcal{E}}:T&\rightarrow&\mathfrak{M}^{ss}_{X^{(1)}}(r\cdot p,d+r(p-1)(g-1))\\
t&\mapsto&\/[{F_{X/k}}_*(\mathcal{E}_t)\/]
\end{eqnarray*}
is a morphism.

By the arbitrariness of $T$ and $\mathcal{E}$ and the universal property of moduli space $\mathfrak{M}^{ss}_X(r,d)$,
the set-theoretic map
\begin{eqnarray*}
S^{ss}_{\mathrm{Frob}}:\mathfrak{M}^{ss}_X(r,d)&\rightarrow&\mathfrak{M}^{ss}_{X^{(1)}}(r\cdot p,d+r(p-1)(g-1))\\
\/[\E\/]&\mapsto&\/[{F_{X/k}}_*(\E)\/]
\end{eqnarray*}
is a morphism. Since $\mathfrak{M}^{ss}_X(r,d)$ and $\mathfrak{M}^{ss}_{X^{(1)}}(r\cdot p,d+r(p-1)(g-1))$ are
projective varieties, it follows that $S^{ss}_{\mathrm{Frob}}$ is a proper morphism.

$(2)$. If $g\geq 2$, then by \cite[Theorem 2.2]{Sun08} we have
$$S_{\mathrm{Frob}}(\mathfrak{M}^{s}_X(r,d))\subseteq\mathfrak{M}^{s}_{X^{(1)}}(r\cdot p,d+r(p-1)(g-1)).$$
On the other hand, we claim that for any vector bundle $\E\in\mathfrak{Coh}(X)$, the
(semi)-stability of ${F_{X/k}}_*(\E)$ implies the (semi)-stability of $\E$. In fact,
let $\F\subset\E$ be a coherent sub-sheaf with $0<\mathrm{rk}(\F)<\mathrm{rk}(\E)$,
then ${F_{X/k}}_*(\F)\subset{F_{X/k}}_*(\E)$ and
$0<\mathrm{rk}({F_{X/k}}_*(\F))<\mathrm{rk}({F_{X/k}}_*(\E))$,
since ${F_{X/k}}_*$ is a left exact functor. As (cf Lemma 4.2 in \cite{Sun08})
$$\mu({F_{X/k}}_*(\E))=\frac{1}{p}\mu(F^*_{X/k}{F_{X/k}}_*(\E))
=\frac{(p-1)(2\cdot g-2)}{2p}+\frac{\mu(\E)}{p}.$$ Thus
$\mu({F_X}_*(\F))<(\mathrm{resp.}\,\leq)\mu({F_X}_*(\E))$ implies
$\mu(\F)<(\mathrm{resp.}\,\leq)\mu(\E)$.
Hence, $${S^{ss^{-1}}_{\mathrm{Frob}}}(\mathfrak{M}^{s}_{X^{(1)}}(r\cdot p,d+r(p-1)(g-1)))=\mathfrak{M}^{s}_X(r,d).$$
Thus the properness of morphism
$S^s_{\mathrm{Frob}}:\mathfrak{M}^s_X(r,d)\rightarrow\mathfrak{M}^s_{X^{(1)}}(r\cdot p,d+r(p-1)(g-1))$
follows from the properness of $S^{ss}_{\mathrm{Frob}}$.
\end{proof}

The following lemma assert that push-forwards preserves the determinant of vector bundles on curves.

\begin{Lemma}\label{Lem:FrobDet}
Let $k$ be an algebraically closed field, $f:X\rightarrow Y$ a finite morphism of smooth projective curves over $k$.
Let $\E$ be a vector bundle, $\mathrm{det}(\E)=\Ox_X(\sum n_iP_i)$, where $P_i\in X$,
$n_i\in\mathbb{Z}$. Then
$$\mathrm{det}(f_*(\E))\cong
(\mathrm{det}(f_*\Ox_X))^{\otimes\mathrm{rk}(\E)}
\otimes_{\Ox_Y}\Ox_Y(\sum n_if(P_i)).$$
In particulary, let $\E_1$ and $\E_2$ be vector bundles such that
$\mathrm{rk}(\E_1)=\mathrm{rk}(\E_2)$ and $\mathrm{det}(\E_1)=\mathrm{det}(\E_2)$.
Then $\mathrm{det}(f_*(\E_1))=\mathrm{det}(f_*(\E_2))$.
\end{Lemma}
\begin{proof} We use induction on the rank of $\E$. Suppose $\mathrm{rk}(\E)=r$.
In the case $r=1$, by \cite[Excise IV.2.6]{Hartshorne77} we have
$$\mathrm{det}(f_*(\E))\cong
(\mathrm{det}(f_*\Ox_X))\otimes_{\Ox_Y}\Ox_Y(\sum n_if(P_i)).$$
The lemma is true. Suppose the lemma is true for any vector bundles of rank less than $r$.
Choose any sub-line bundle $\mathscr{L}\subset\E$ such that $\E/\mathscr{L}$
is also a vector bundle. Consider the exact sequence of $\Ox_Y$-modules
$$0\rightarrow f_*(\mathscr{L})\rightarrow
f_*(\E)\rightarrow f_*(\E/\mathscr{L})\rightarrow 0,$$
then we have $\mathrm{det}(f_*(\E))\cong\mathrm{det}(f_*(\mathscr{L}))\otimes_{\Ox_Y}
\mathrm{det}(f_*(\E/\mathscr{L}))$.
On the other hand, since $\mathrm{det}(\E)\cong
\mathrm{det}(\mathscr{L})\otimes_{\Ox_Y}\mathrm{det}(\E/\mathscr{L})$,
by induction hypothesis we have
$$\mathrm{det}(f_*(\E))\cong
(\mathrm{det}(f_*\Ox_X))^{\otimes\mathrm{rk}(\E)}
\otimes_{\Ox_Y}\Ox_Y(\sum n_if(P_i)),$$
where $\mathrm{det}(\E)=\Ox_X(\sum n_iP_i)$, $P_i\in X$, $n_i\in\mathbb{Z}$.
\end{proof}

\begin{Corollary} Let $k$ be an algebraically closed field of characteristic $p>0$, and $X$
a smooth projective curve of genus $g\geq 1$ over $k$, $\mathscr{L}\in\mathrm{Pic}(X)$.
Then the set-theoretic map
\begin{eqnarray*}
S^{\mathscr{L}}_{\mathrm{Frob}}:\mathfrak{M}^{ss}_X(r,\mathscr{L})&\rightarrow&
\mathfrak{M}^{ss}_{X^{(1)}}(r\cdot p,\mathrm{det}({F_{X/k}}_*(\E))),
(\forall\E\in\mathfrak{M}^{ss}_X(r,\mathscr{L}))\\
\/[\E\/]&\mapsto&\/[{F_{X/k}}_*(\E)\/]
\end{eqnarray*}
is a proper morphism.
\end{Corollary}

\begin{proof} For any $[\E]\in\mathfrak{M}^{ss}_X(r,\mathscr{L})$. Since $\mathfrak{M}^{ss}_X(r,\mathscr{L})$
and $\mathfrak{M}^{ss}_{X^{(1)}}(r\cdot p,\mathrm{det}({F_{X/k}}_*(\E)))$ are closed sub-varieties of
$\mathfrak{M}^{ss}_{X^{(1)}}(r,d)$ and $\mathfrak{M}^{ss}_{X^{(1)}}(r\cdot p,d+r(p-1)(g-1))$ respectively.
Then the corollary follows from Proposition \ref{Prop:SunMorp} and Lemma \ref{Lem:FrobDet}.
\end{proof}

We now use the functoriality of canonical filtration and the uniqueness of Harder-Narasimhan filtration to prove
the following theorem.

\begin{Theorem}\label{Thm:Tangmap} Let $k$ be an algebraically closed field of characteristic $p>0$, and $X$
a smooth projective variety with a fixed ample divisor $H$. Let
$\E_1$, $\E_2$ be slope semi-stable vector bundles
such that
$\E_i\otimes_{\Ox_X}\mathrm{T}^l(\Omg^1_{X/k})$ are
slope semi-stable for any integer $0\leq l\leq n(p-1)$, $(i=1,2)$.
If $\mu(\Omg^1_{X/k})>0$. Then
\begin{itemize}
    \item[$(1)$.] The canonical filtration of $F^*_{X/k}{F_{X/k}}_*(\E_i)$
    is precisely the Harder-Narasimhan filtration of
    $F^*_{X/k}{F_{X/k}}_*(\E_i)$, $(i=1,2)$. In particular,
    ${F_{X/k}}_*(\E_1)\cong{F_{X/k}}_*(\E_2)$ implies $\E_1\cong\E_2$.
    \item[$(2)$.] If $\mu(\E_1)=\mu(\E_2)$, Then the natural $k$-linear homomorphism
    \begin{eqnarray*}
    \Phi:\mathrm{Ext}^1_X(\E_1,\E_2)&\rightarrow&\mathrm{Ext}^1_{X^{(1)}}({F_{X/k}}_*(\E_1),{F_{X/k}}_*(\E_2))\\
    \/[0\rightarrow\E_1\rightarrow\F\rightarrow\E_2\rightarrow 0\/]&\mapsto&\/[0\rightarrow{F_{X/k}}_*(\E_1)\rightarrow{F_{X/k}}_*(\F)
        \rightarrow{F_{X/k}}_*(\E_2)\rightarrow 0\/]
    \end{eqnarray*}
    is an injective homomorphism.
\end{itemize}
\end{Theorem}

\begin{proof} $(1)$. Consider the canonical filtration of $F^*_{X/k}{F_{X/k}}_*(\E_i)$:
$${\mathbb{F}^{\mathrm{can}}_{\E_i}}_\bullet:0=V^{\E_i}_{n(p-1)+1}\subset V^{\E_i}_{n(p-1)}
\subset\cdots V^{\E_i}_{l+1}\subset
V^{\E_i}_l\cdots\subset V^{\E_i}_1 \subset
V^{\E_i}_0=F^*_{X/k}{F_{X/k}}_*(\E_i).$$ Since
$V^{\E_i}_l/V^{\E_i}_{l+1}\cong
\E_i\otimes_{\Ox_X}\mathrm{T}^l(\Omg^1_{X/k})(0\leq
l\leq n(p-1))$ are slope semi-stable vector bundles, and for any
integer $0\leq s<t\leq n(p-1)$,
\begin{eqnarray*}
\mu(\E_i\otimes_{\Ox_X}\mathrm{T}^s(\Omg^1_{X/k}))&=&\mu(\E_i)+s\cdot\mu(\Omg^1_{X/k})\\
&<&\mu(\E_i)+t\cdot\mu(\Omg^1_{X/k})\\
&=&\mu(\E_i\otimes_{\Ox_X}\mathrm{T}^t(\Omg^1_{X/k})).
\end{eqnarray*}
Then by the uniqueness of Harder-Narasimhan filtration we know that
canonical filtration
${\mathbb{F}^{\mathrm{can}}_{\E_i}}_\bullet$ is precisely
the Harder-Narasimhan filtration of
$F^*_{X/k}{F_{X/k}}_*(\E_i)$. It is easy to see that
${F_{X/k}}_*(\E_1)\cong{F_{X/k}}_*(\E_2)$ implies
$\E_1\cong\E_2$ by the uniqueness of
Harder-Narasimhan filtration.

$(2)$. Let $e_i:=[0\rightarrow\E_1\rightarrow
\F_i\rightarrow\E_2\rightarrow
0]\in\mathrm{Ext}^1_X(\E_1,\E_2)$, $i=1,2$, such
that $\Phi(e_1)=\Phi(e_2)$, i.e. there exists isomorphism of
$\Ox_{X^{(1)}}$-modules
$\phi:{F_{X/k}}_*(\F_1)\stackrel{\cong}{\rightarrow}{F_{X/k}}_*(\F_2)$
such that the following diagram
$$\xymatrix{
  0 \ar[r] & {F_{X/k}}_*(\E_1) \ar[r]\ar@{=}[d] & {F_{X/k}}_*(\F_1) \ar[r]\ar[d]_{\phi}^{\cong}
  & {F_{X/k}}_*(\E_2) \ar[r]\ar@{=}[d] & 0\\
  0 \ar[r] & {F_{X/k}}_*(\E_1) \ar[r] & {F_{X/k}}_*(\F_2) \ar[r] &
  {F_{X/k}}_*(\E_2) \ar[r] & 0.}$$
of $\Ox_{X^{(1)}}$-modules is commutative. Taking
$F^*_{X/k}$ to the above exact sequence, we get the following
commutative diagram of $\Ox_X$-modules
$$\xymatrix{
  0 \ar[r] & F^*_{X/k}{F_{X/k}}_*(\E_1) \ar[r]\ar@{=}[d] & F^*_{X/k}{F_{X/k}}_*(\F_1) \ar[r]\ar[d]_{F^*_{X/k}(\phi)}^{\cong} & F^*_{X/k}{F_{X/k}}_*(\E_2) \ar[r]\ar@{=}[d] & 0\\
  0 \ar[r] & F^*_{X/k}{F_{X/k}}_*(\E_1) \ar[r] & F^*_{X/k}{F_{X/k}}_*(\F_2) \ar[r] & F^*_{X/k}{F_{X/k}}_*(\E_2) \ar[r] & 0.}$$
On the other hand, for any integer $0\leq l\leq n(p-1)$, consider
the exact sequence of $\Ox_X$-modules
$0\rightarrow\E_1\otimes_{\Ox_X}\mathrm{T}^l(\Omg^1_{X/k})\rightarrow\F_i\otimes_{\Ox_X}\mathrm{T}^l(\Omg^1_{X/k})\rightarrow\E_2\otimes_{\Ox_X}\mathrm{T}^l(\Omg^1_{X/k})\longrightarrow0.$ Because
$\E_j\otimes_{\Ox_X}\mathrm{T}^l(\Omg^1_{X/k})(j=1,2)$
are slope semi-stable bundles and
$\mu(\E_1)=\mu(\E_2)$, then
$\F_i\otimes_{\Ox_X}\mathrm{T}^l(\Omg^1_{X/k})(i=1,2)$
are also slope semi-stables. It follows from $(1)$ that the
canonical filtration
${\mathbb{F}^{\mathrm{can}}_{\F_i}}_\bullet$ of
$F^*_{X/k}{F_{X/k}}_*(\F_i)$ is also coincided with the
Harder-Narasimhan filtration of
$F^*_{X/k}{F_{X/k}}_*(\F_i)$.

Consider the canonical filtration
${\mathbb{F}^{\mathrm{can}}_{\E_i}}_\bullet$,
${\mathbb{F}^{\mathrm{can}}_{\F_i}}_\bullet$ of
$F^*_{X/k}{F_{X/k}}_*(\E_i)$ and
$F^*_{X/k}{F_{X/k}}_*(\F_i)$ respectively, $i=1,2$ (which
are coincided with Harder-Narasimhan filtration). Then we have the
following commutative diagram
$$\xymatrix@C=0.4cm@R=0.5cm{
  0 \ar[rr] && F^*_{X/k}{F_{X/k}}_*(\E_1) \ar[rr]\ar@{=}'[d][dd]\ar[dr] && F^*_{X/k}{F_{X/k}}_*(\F_1) \ar[rr]\ar'[d]_-{F^*_{X/k}(\phi)}^-{\cong}[dd]\ar[dr] && F^*_{X/k}{F_{X/k}}_*(\E_2) \ar[rr]\ar@{=}'[d][dd]\ar[dr] && 0\\
  & 0 \ar[rr] && \E_1 \ar[rr]\ar@{=}[dd] && \F_1 \ar[rr]\ar@{-->}[dd]_{?\exists} && \E_2 \ar[rr]\ar@{=}[dd] && 0 &\\
  0 \ar[rr] && F^*_{X/k}{F_{X/k}}_*(\E_1) \ar'[r][rr]\ar[dr] && F^*_{X/k}{F_{X/k}}_*(\F_2) \ar'[r][rr]\ar[dr] && F^*_{X/k}{F_{X/k}}_*(\E_2) \ar'[r][rr]\ar[dr] && 0\\
  & 0 \ar[rr] && \E_1 \ar[rr] && \F_2 \ar[rr] && \E_2 \ar[rr] && 0 .&\\
  }$$
Restricting the isomorphism $F^*_{X/k}(\phi):F^*_{X/k}{F_{X/k}}_*(\F_1)\rightarrow
F^*_{X/k}{F_{X/k}}_*(\F_2)$ to $V^{\F_i}_1$, we
get an isomorphism
$F^*_{X/k}(\phi)|_{V^{\F_i}_1}:V^{\F_i}_1\rightarrow
V^{\F_2}_1$, which induces a natural isomorphism
$\psi:\F_1\stackrel{\cong}{\rightarrow}\F_2$ with following commutative diagram
$$\xymatrix{
  0 \ar[r] & \E_1 \ar[r]\ar@{=}[d] & \F_1 \ar[r]\ar[d]_{\psi}^{\cong} & \E_2 \ar[r]\ar@{=}[d] & 0\\
  0 \ar[r] & \E_1 \ar[r] & \F_2 \ar[r] & \E_2 \ar[r] & 0.}$$
Hence $\Phi:\mathrm{Ext}^1_X(\E_1,\E_2)\rightarrow\mathrm{Ext}^1_{X^{(1)}}({F_{X/k}}_*(\E_1),{F_{X/k}}_*(\E_2))$
is an injection.
\end{proof}

\begin{Theorem}\label{Lee}
Let $k$ be an algebraically closed field of characteristic $p>0$, and $X$ a smooth projective curve of genus $g\geq 2$ over $k$. Then the morphism $$S^s_{\mathrm{Frob}}:\mathfrak{M}^{s}_X(r,d)\rightarrow\mathfrak{M}^{s}_{X^{(1)}}(r\cdot p,d+r(p-1)(g-1))$$
is a closed immersion. In particular, for any $\mathscr{L}\in\mathrm{Pic}(X)$, the morphism
$$S^{s,\mathscr{L}}_{\mathrm{Frob}}:\mathfrak{M}^{s}_X(r,\mathscr{L})\rightarrow\mathfrak{M}^{s}_{X^{(1)}}(r\cdot p,\mathrm{det}({F_{X/k}}_*(\E))),(\forall\E\in\mathfrak{M}^{ss}_X(r,\mathscr{L}))$$
is also a closed immersion.
\end{Theorem}
\begin{proof}Since $\dim X=1$, $\mathrm{T}^l(\Omg^1_{X/k})\cong(\Omg^1_{X/k})^{\otimes l}$ for any integer $0\leq l\leq n(p-1)$. Then for any stable bundle $\E$ and any integer $0\leq l\leq n(p-1)$, $\E\otimes_{\Ox_X}\mathrm{T}^l(\Omg^1_{X/k})$ are stable. Therefore, by Proposition\ref{Prop:SunMorp} and Theorem\ref{Thm:Tangmap}, the morphism $$S^s_{\mathrm{Frob}}:\mathfrak{M}^{s}_X(r,d)\rightarrow\mathfrak{M}^{s}_{X^{(1)}}(r\cdot p,d+r(p-1)(g-1))$$ is an injective proper morphism. Since for any $[\E]\in\mathfrak{M}^{s}_X(r,d)$, ${F_{X/k}}_*(\E)$ is also stable by \cite[Theorem 2.2]{Sun08}. Then the tangent space of $\mathfrak{M}^{s}_X(r,d)$ at $[\E]$ is $$T_{[\E]}\mathfrak{M}^{s}_X(r,d)=\mathrm{Ext}^1_X(\E,\E),$$
and tangent space of $\mathfrak{M}^{s}_{X^{(1)}}(r\cdot p,d+r(p-1)(g-1))$ at $[{F_{X/k}}_*(\E)]$ is $$T_{[{F_{X/k}}_*(\E)]}\mathfrak{M}^{s}_{X^{(1)}}(r\cdot p,d+r(p-1)(g-1))=\mathrm{Ext}^1_X({F_{X/k}}_*(\E),{F_{X/k}}_*(\E)).$$
Moreover, tangent map of $S^s_{\mathrm{Frob}}$ at $[\E]$ is precisely the homomorphism
\begin{eqnarray*}
T_{S^s_{\mathrm{Frob}},[\E]}:\mathrm{Ext}^1_X(\E,\E)&\rightarrow&\mathrm{Ext}^1_X({F_{X/k}}_*(\E),{F_{X/k}}_*(\E))\\
\/[0\rightarrow\E\rightarrow\F\rightarrow\E\rightarrow 0\/]&\mapsto&\/[0\rightarrow{F_{X/k}}_*(\E)\rightarrow{F_{X/k}}_*(\F)\rightarrow{F_{X/k}}_*(\E)\rightarrow 0\/]
\end{eqnarray*}
Then by Theorem \ref{Thm:Tangmap} we have the tangent map $T_{S^s_{\mathrm{Frob}},[\E]}$ is an injective homomorphism. Therefore the morphism $S^s_{\mathrm{Frob}}$ is a closed immersion by the well known criterion of closed immersion.
\end{proof}

\section{The Stratum of Frobenius Stratification with Maximal Harder-Narasimhan Polygon}

\begin{Definition}\label{Def:LocalSystem}
Let $k$ be an algebraically closed field, $X$ a smooth projective curve over $k$. A \emph{local system} on $X$ is pair $(\E,\nabla)$ consists of a vector bundle $\E$ on $X$ and a connection $\nabla$ on $X$. A local system $(\E,\nabla)$ is called \emph{semi-stable} (resp. \emph{stable}) if any proper $\nabla$-invariant subbundle $\F\subset\E$ satisfies $\mu(\F)\leq\E$ (resp. $\mu(\F)<\E$). An \emph{oper} on $X$ is a local sytem $(\E,\nabla)$ together with a filtration of subbundles $\E_\bullet$: $0=\E_m\subset\E_{m-1}\subset\cdots\subset\E_1\subset\E_0=\E$, such that
\begin{itemize}
\item[$(i)$.] $\nabla(\E_i)\subseteq\E_{i-1}\otimes_{\Ox_X}\Omg_{X/k}$ for any $1\leq i\leq m-1$;
\item[$(ii)$.] $\E_i/\E_{i+1}\stackrel{\nabla}{\rightarrow}(\E_{i-1}/\E_i)\otimes_{\Ox_X}\Omg_{X/k}$ is an isomorphisms for any $1\leq i\leq m-1$.
\end{itemize}
\end{Definition}

Let $r\in\Z_{>0}$ and $d\in\Z$ such that $r|d$. We introduce the \emph{oper-polygon} for pair $(r,d)$ in the plane $\R^2$
$$\Pg_{r,d}^{oper}:~\text{with vertices}~(i,i\frac{d}{r}+i(r-i)(g-1))~\text{for}~0\leq i\leq r.$$
$\Pg_{r,0}^{oper}$ is simply denoted by $\Pg_r^{oper}$ (See also section 5.3 of \cite{JoshiPauly09}).

\begin{Remark}\label{Remark:OperPgn}
If $g\geq 2$, then the Harder-Narasimhan polygon of any oper $(\E,\nabla,\E_\bullet)$ of rank $r$, degree $d$ and type $1$ (i.e. $\rank(\E_i/\E_{i+1})=1$ for $0\leq i\leq r-1$) is $\Pg_{r,d}^{oper}$. Since in this case the Harder-Narasimhan filtration of $\E$ is just $\E_\bullet$.
\end{Remark}

\begin{Lemma}\cite[Lemma 5.1.1]{JoshiPauly09}\label{Lemma:InstLS}
Let $k$ be an algebraically closed field, $X$ a smooth projective curve of genus $g$ over $k$. Let $(\E,\nabla)$ be a semi-satble local system with Harder-Narasimhan filtration $0=\E_m\subset\E_{m-1}\subset\cdots\subset \E_1\subset \E_0=\E$.
Then for nay $1\leq i\leq m-1$, $\mu(\E_i/\E_{i-1})-\mu(\E_{i+1}/\E_i)\leq 2g-2$.
\end{Lemma}

\begin{Proposition}\label{Prop:MaxLS}
Let $k$ be an algebraically closed field, $X$ a smooth projective curve of genus $g$ over $k$, $r\in\Z_{>0}$, $d\in\Z$ such that $r|d$. Let $(\E,\nabla)$ be a semi-stable local system $(\E,\nabla)$ of rank $r$ and degree $d$ on $X$. Then
\begin{itemize}
\item[$(1)$.] $\Pg_{r,d}^{oper}\succcurlyeq\HNP(\E)$.
\item[$(2)$.] $\Pg_{r,d}^{oper}=\HNP(\E)$ if and only if $(\E,\nabla,\E^{\mathrm{HN}}_{\bullet})$ is an oper of type $1$, where $\E^{\mathrm{HN}}_{\bullet}$ is the Harder-Narasimhan filtration of $\E$.
\end{itemize}
\end{Proposition}

\begin{proof}$(1)$. Let $0=\E_m\subset\E_{m-1}\subset\cdots\subset\E_1\subset \E_0=\E$ be the Harder-Narasimhan filtration of $\E$. Suppose that $\Pg_{r,d}^{oper}\not\succcurlyeq\HNP(\E)$. Then there exists some $0\leq j\leq m-1$ such that the point $(\rank(\E_{j+1}),\degree(\E_{j+1}))$ lies on or blow the $\Pg_{r,d}^{oper}$ and the point $(\rank(\E_j),\degree(\E_j))$ lies above the $\Pg_{r,d}^{oper}$, i.e. there exist $0\leq j\leq m-1$ and $1\leq i\leq r$ with the properties
$\mu(\E_j/\E_{j+1})>\frac{d}{r}+(r-2i+1)(g-1)$ and $\degree(\E_j)>i\frac{d}{r}+i(r-i)(g-1)$.
Then by Lemma \ref{Lemma:InstLS}, we have
\begin{eqnarray*}
\degree(\E)&=&\degree(\E_j)+\sum\limits_{l=0}^{j-1}\rank(\E_l/\E_{l+1})\cdot\mu(\E_l/\E_{l+1})\\
&\geq&\degree(\E_j)+\sum\limits_{l=0}^{j-1}\rank(\E_l/\E_{l+1})\cdot[\mu(\E_j/\E_{j+1})-(j-l)(g-1)]\\
&\geq&\degree(\E_j)+\sum\limits_{l=0}^{r-i}[\mu(\E_j/\E_{j+1})-(r-i-l+1)(g-1)]\\
&>&d.
\end{eqnarray*}
This contradict to the fact $\degree(\E)=d$. Hence $\Pg_{r,d}^{oper}\succcurlyeq\HNP(\E)$.

$(2)$. Note that the part two of \cite[Theorem 5.3.1]{JoshiPauly09} make the assumption that $d=0$. In fact, their proof is also valid for any $d$ such that $r|d$.
\end{proof}

\begin{Theorem}\label{Thm:Eqiv}
Let $k$ be an algebraically closed field of characteristic $p>0$, and $X$ a smooth projective curve of genus $g\geq 2$ over $k$. Let $\E$ be a stable vector bundle on $X^{(1)}$ of rank $p$ and degree $d$. Then the following conditions are equivalent
\begin{itemize}
\item[$(i)$.] $\HNP(F^*_{X/k}(\E))=\Pg_{p,pd}^{oper}$.
\item[$(ii)$.] $\mu_{\max}(F^*_{X/k}(\E))-\mu_{\min}(F^*_{X/k}(\E))=(p-1)(2g-2)$.
\item[$(iii)$.] $\E\cong{F_{X/k}}_*(\mathscr{L})$ for some line bundle $\mathscr{L}$ on $X$.
\item[$(iv)$.] $\HNP(F^*_{X/k}(\E))\succcurlyeq\HNP(F^*_{X/k}(\F))$ for any $\F\in\mathfrak{M}^{s}_{X^{(1)}}(p,d)$.
\end{itemize}
\end{Theorem}

\begin{proof} $(i)\Rightarrow(ii)$ is obvious by definition of $\Pg_{p,pd}^{oper}$.

$(ii)\Leftrightarrow(iii)$ follows from \cite[Proposition 1.4]{MehtaPauly07}.

$(iii)\Rightarrow(iv)$. By \cite[Theorem 5.3]{JRXY06}, we have $(F^*_{X/k}{F_{X/k}}_*(\mathscr{L}),\nabla_{\mathrm{can}})$ is an oper of type $1$. It follows from Remark \ref{Remark:OperPgn} that $\HNP(F^*_{X/k}{F_{X/k}}_*(\mathscr{L}))=\Pg_{p,pd}^{oper}$. Then $\HNP(F^*_{X/k}{F_{X/k}}_*(\mathscr{L}))\succcurlyeq\HNP(F^*_{X/k}(\F))$ for any $\F\in\mathfrak{M}^{s}_{X^{(1)}}(p,d)$ by Proposition \ref{Prop:MaxLS}, since $(F^*_{X/k}(\F),\nabla_{\mathrm{can}})$ is a semi-stable local system by Cartier's theorem (\cite[Theorem 5.1]{Katz70}).

$(iv)\Rightarrow(i)$. For any $\F\in\mathfrak{M}^{s}_{X^{(1)}}(p,d)$, $(F^*_{X/k}(\F),\nabla_{\mathrm{can}})$ is a semi-stable local system. Then $\Pg_{p,pd}^{oper}\succcurlyeq\HNP(F^*_{X/k}(\F))$ by Lemma \ref{Prop:MaxLS}. On the other hand, choose any line bundle $\mathscr{L}$ of degree $d-(p-1)(g-1)$, then ${F_{X/k}}_*(\mathscr{L})$ is a stable vector bundle of rank $p$ and degree $d$ and $\HNP(F^*_{X/k}{F_{X/k}}_*(\mathscr{L}))=\Pg_{p,pd}^{oper}$. Hence if $\HNP(F^*_{X/k}(\E))\succcurlyeq\HNP(F^*_{X/k}(\F))$ for any $\F\in\mathfrak{M}^{s}_{X^{(1)}}(p,d)$, then $\HNP(F^*_{X/k}(\E))=\Pg_{p,pd}^{oper}$.
\end{proof}

\begin{Corollary}\label{MaxFrobStr}
Let $k$ be an algebraically closed field of characteristic $p>0$, and $X$ a smooth projective curve of genus $g\geq 2$ over $k$. Then the subset $$W=\{\E\in\mathfrak{M}^{s}_{X^{(1)}}(p,d)~|~\HNP(F^*_{X/k}(\E))\succcurlyeq\HNP(F^*_{X/k}(\F))~\text{for any}~\F\in\mathfrak{M}^{s}_{X^{(1)}}(p,d)\}$$
is a closed sub-variety of $\mathfrak{M}^{s}_{X^{(1)}}(p,d)$, which is isomorphic to Jacobian variety $\Jac_X$ of $X$.
In particular, $W$ is an irreducible smooth projective variety of dimension $g$.
\end{Corollary}

\begin{proof}By Theorem \ref{Thm:Eqiv}, we know that $W$ is precisely the image of the morphism $$S^s_{\mathrm{Frob}}:\mathfrak{M}^{s}_X(1,d-(p-1)(g-1))\rightarrow\mathfrak{M}^{s}_{X^{(1)}}(p,d).$$
Then this Corollary follows from Theorem \ref{Lee} and the trivial fact $$\mathfrak{M}^{s}_X(1,-(p-1)(g-1))\cong\Jac_X.$$
\end{proof}

In the case of $(p,r)=(2,2)$, K. Joshi, S. Ramanan, E. Z. Xia and J.-K. Yu \cite[Theorem 4.6.4]{JRXY06} show that the locus of moduli space $\mathfrak{M}^{s}_{X^{(1)}}(2,d)$ consist of stable vector bundles $\E$ whose Frobenius pull back $F^*_{X/k}(\E)$ have maximal Harder-Narasimhan Polygon is a irreducible projective variety of dimension $g$. The Corollary \ref{MaxFrobStr} generalize this result to more general case via different method.

\section{Acknowledgments}

I would like to express my hearty thanks to my advisor Professor X. Sun for helpful discussions.

\end{document}